\documentclass[a4paper,11pt]{amsart}
\usepackage{amssymb}
\usepackage{ifthen}
\usepackage{cite}
 \usepackage[dvips]{graphicx}

\nonstopmode \numberwithin{equation}{section}
\usepackage{amssymb}
\usepackage{ifthen}
\usepackage{graphicx}
\usepackage{amsmath}
\usepackage[T1]{fontenc} 
\usepackage[utf8]{inputenc}
\usepackage[usenames,dvipsnames]{color}
\usepackage{color}
\usepackage[english]{babel}
\usepackage{fancyhdr}
\usepackage{fancybox}
\usepackage{tikz}

\nonstopmode \numberwithin{equation}{section}
\theoremstyle{plain}

\newtheorem{conj}{Conjecture}

\theoremstyle{definition}
\newtheorem{defn}{Definition}[section]

\newtheorem{thm}{Theorem}[section]
\newtheorem{prob}{Problem}[section]
\newtheorem{cor}{Corollary}[section]

\newtheorem{prop}{Proposition}[section]
\newtheorem{rem}{Remark}[section]
\newtheorem{lem}{Lemma}[section]
\newtheorem{defi}{Definition}[section]

\newtheorem*{thmA}{Theorem A}
\newtheorem*{thmB}{Theorem B}

\newtheorem*{lemA}{Lemma A}
\newtheorem*{lemB}{Lemma B}

\newcounter{minutes}
\divide\time by 60
\newcounter{hours}
\multiply\time by 60
\addtocounter{minutes}{-\time}

\newcounter {own}
\def\theown {\thesection       .\arabic{own}}

\newenvironment{pf}[1][]{%
 \vskip 3mm
 \noindent
 \ifthenelse{\equal{#1}{}}%
  {{\slshape Proof. }}%
  {{\slshape #1.} }%
 }%
{\qed\bigskip}

\newcounter{alphabet}

\def\be{\begin{equation}}
\def\ee{\end{equation}}

\newcommand{\bee}{\begin{enumerate}}
\newcommand{\eee}{\end{enumerate}}

\newcommand{\blem}{\begin{lem}}
\newcommand{\elem}{\end{lem}}
\newcommand{\bthm}{\begin{thm}}
\newcommand{\ethm}{\end{thm}}
\newcommand{\bcor}{\begin{cor}}
\newcommand{\ecor}{\end{cor}}
\newcommand{\beg}{\begin{examp}}
\newcommand{\eeg}{\end{examp}}
\newcommand{\begs}{\begin{examples}}
\newcommand{\eegs}{\end{examples}}

\newcommand{\bdefn}{\begin{defn}}
\newcommand{\edefn}{\end{defn}}

\newcommand{\bprob}{\begin{prob}}
\newcommand{\eprob}{\end{prob}}
\newcommand{\bei}{\begin{itemize}}
\newcommand{\eei}{\end{itemize}}

\newcommand{\bcon}{\begin{conj}}
\newcommand{\econ}{\end{conj}}
\newcommand{\bcons}{\begin{conjs}}
\newcommand{\econs}{\end{conjs}}
\newcommand{\bprop}{\begin{prop}}
\newcommand{\eprop}{\end{prop}}
\newcommand{\br}{\begin{rem}}
\newcommand{\er}{\end{rem}}
\newcommand{\brs}{\begin{rems}}
\newcommand{\ers}{\end{rems}}
\newcommand{\bo}{\begin{obser}}
\newcommand{\eo}{\end{obser}}
\newcommand{\bos}{\begin{obsers}}
\newcommand{\eos}{\end{obsers}}
\newcommand{\bpf}{\begin{pf}}
\newcommand{\epf}{\end{pf}}
\newcommand{\ba}{\begin{array}}
\newcommand{\ea}{\end{array}}
\newcommand{\beq}{\begin{eqnarray}}
\newcommand{\beqq}{\begin{eqnarray*}}
\newcommand{\eeq}{\end{eqnarray}}
\newcommand{\eeqq}{\end{eqnarray*}}

\begin{document}

\title{Coefficient estimates and Bohr phenomenon for analytic functions involving semigroup generator}

\author{Molla Basir Ahamed}
\address{Molla Basir Ahamed, Department of Mathematics, Jadavpur University, Kolkata-700032, West Bengal, India.}
\email{mbahamed.math@jadavpuruniversity.in}

\author{Sanju Mandal}
\address{Sanju Mandal, Department of Mathematics, Jadavpur University, Kolkata-700032, West Bengal, India.}
\email{sanju.math.rs@gmail.com, sanjum.math.rs@jadavpuruniversity.in}

\subjclass[{AMS} Subject Classification:]{30A10, 30H05, 30C35, 30C45, 30C50}

\def\thefootnote{}
\footnotetext{ {\tiny File:~\jobname.tex,
printed: \number\year-\number\month-\number\day,
          \thehours.\ifnum\theminutes<10{0}\fi\theminutes }
} \makeatletter\def\thefootnote{\@arabic\c@footnote}\makeatother

\begin{abstract} 
This article investigates the Bohr phenomenon and sharp coefficient problems for the class $\mathcal{A}_{\beta}$, a subclass of analytic self-maps of the unit disk with the holomorphic generators of one-parameter continuous semigroups. By integrating concepts from complex dynamics and geometric function theory, we derive sharp improvements to the classical Bohr radius by incorporating multiple Schwarz functions and certain functional expressions. We establish generalized versions of the Bohr and Bohr-Rogosinski inequalities and determine the best possible radii for these refinements. Furthermore, we provide a sharp solution to the classical Fekete-Szeg\"o problem for the class $\mathcal{A}_{\beta}$ by obtaining sharp bounds for the functional $|a_3 - \mu a_2^2|$ for all real values of $\mu$. Additionally, we derive sharp inequalities for the moduli of differences of logarithmic coefficients for both the functions and their inverses in this class. 
\end{abstract}

\maketitle
\pagestyle{myheadings}
\markboth{M. B. Ahamed and S. Mandal}{Coefficient problems and Bohr phenomenon}

\section{\bf Introduction}
Let $\mathcal{H}$ be the class of holomorphic functions in the unit disk $\mathbb{D}:=\{z\in\mathbb{C} : |z|<1\}$ and $\mathcal{A}\subset \mathcal{H}$ containing functions having properties $f(0)=0$ and $f'(0)=1$ \textit{i.e.}, of the form
\begin{align}\label{Eq-2.1}
	f(z)=z+\sum_{n=2}^{\infty}a_nz^n.
\end{align}
Let $\mathcal{B}$ be the class of holomorphic self mappings from $\mathbb{D}$ to $\mathbb{D}$. A family $\{u_t(z)\}_{t\geq 0}\subset\mathcal{B}$ is called a one-parameter continuous semigroup if 
\begin{enumerate}
	\item[(i)] $\lim\limits_{t\to 0} u_t(z)=z$.
	\item[(ii)] $u_{t+s}(z)=u_t(z)u_s(z)$,
	\item[(iii)]  $\lim\limits_{t\to s}u_t(z)=u_s(z)$
\end{enumerate}
for each $z\in\mathbb{D}$ hold.\vspace{1.2mm}

In $1978$, Berkson and Potra \cite{Berkson-Porta-MMJ-1978} shows that each one-parameter semigroup is locally differentiable in parameter $t\geq 0$ and moreover, if 
\begin{align*}
	\lim\limits_{t\to 0}\frac{z-u_t(z)}{t}=f(z),
\end{align*}
which is a holomorphic function, then $u_t(z)$ is the solution of the Cauchy problem 
\begin{align*}
	\frac{\partial u_t(z)}{\partial t}+f(u_t(z))=0,\; u_0(z)=z.
\end{align*}
The function $f$ is called the holomorphic generator of semigroup $\{u_t(z)\}\subset\mathcal{B}$. The class of all holomorphic generators is denoted by $\mathcal{G}$. Also, note that each element of $\{u_t(z)\}$ generated by $f\in\mathcal{G}$ is univalent function while $f$ is not necessarily univalent (see \cite{Elin-Shoiket-2010}). Various properties of generators and semigroup generated by them are established in \cite{Berkson-Porta-MMJ-1978,Bracci-Contreras-Diaz-2020,Elin-Shoiket-2010,Elin-Shoikhet-Sugawa-2018,Elin-Reich-Shoikhet-2019,Shoikhet-2001} and references therein. Berkson and Porta \cite{Berkson-Porta-MMJ-1978} proved the following theorem.
\begin{thmA}\cite{Berkson-Porta-MMJ-1978}
The following assertions are equivalent:
\begin{enumerate}
	\item[(a)] $f\in\mathcal{G}$,
	\item[(b)] $f(z)=(z-\sigma)(1-z\bar{\sigma})p(z)$ with some $\sigma\in\overline{\mathbb{D}}$ and $p\in\mathcal{H}$,\; ${\rm Re}\; p(z)\geq 0$.  
\end{enumerate}
\end{thmA}
The point $\sigma\in\bar{\mathbb{D}}:=\{z\in\mathbb{C} : |z|\leq 1\}$ is called the Denjoy-Wolff point of the semigroup generated by $f$. By Denjoy-Wolff theorem (see \cite{Shoikhet-2001}) for continuous semigroups, if any element of the semigroup generated by $f$ is neither an eleliptic automorphism of $\mathbb{D}$ nor the identity map for at least one $t\in [0, \infty)$, then there is a unique point $\sigma\in\bar{\mathbb{D}}$ such that $\lim\limits_{t\to\infty}u_t(t, z)=\sigma$ uniformly, for each $z\in\mathbb{D}$.\vspace{2mm}

We denote the class of holomorphic generators with Denjoy-Wolff point $\sigma$ by $\mathcal{G}[\sigma]$. For $\sigma=0$, we obtain the subclass 
\begin{align*}
	\mathcal{G}[0]=\{f\in\mathcal{G} : f(z)=zp(z),\; {\rm Re}\;p(z)\geq 0\}.
\end{align*} 
Bracci \emph{et al.} \cite{Bracci et. al.-FACM-2018} considered the class $\mathcal{G}_0=\mathcal{G}[0]\cap \mathcal{A}$. In the study of nonautonomous problem such as Loewner theory, the class $\mathcal{G}_0$ plays a significant role (see \cite{Bracci-Contreras-Diaz-JRAM-2012, Duren-1983}). Various subclasses of $\mathcal{G}_0$ with parameter such that $\mathcal{R}$ is the smallest one were recently studied (also called filtration), where 
\begin{align*}
	\mathcal{R}=\{f\in\mathcal{A} : {\rm Re}\; f'(z)>0\}
\end{align*}
is the class of functions with bounded turning (see \cite{Bracci et. al.-FACM-2018,Elin-Shoikhet-Sugawa-2018,Shoikhet-MJM-2016}). In particular, for $\beta\in [0, 1]$, the class 
\begin{align}
	\mathcal{A}_{\beta}:=\left\{f\in\mathcal{A} : {\rm Re}\left(\beta\frac{f(z)}{z}+(1-\beta)f'(z)>0\right)\right\}
\end{align}
is a subclass of $\mathcal{G}_0$. In \cite{Bracci et. al.-FACM-2018}, it is proved that $\mathcal{A}_{\beta_1}\subsetneq \mathcal{A}_{\beta_2}\subsetneq \mathcal{G}_0$ for $0\leq\beta_1<\beta_2<1$ and whenever $f\in \mathcal{A}_{\beta}$,
\begin{align*}
	{\rm Re}\left(\frac{f(z)}{z}\right)\geq \int_{0}^{1}\frac{1-t^{1-\beta}}{1+t^{1-\beta}}dt.
\end{align*}
Clearly, for the special case $\beta=0$, the class $\mathcal{A}_{\beta}$ reduces to the class $\mathcal{R}$. Elin \emph{et al.} \cite{Elin-Shoikhet-Tuneski-2020} solved the radii problems for $\mathcal{A}_{\beta}$; specifically, they determined the radii $r \in (0, 1)$ such that $f(rz)/r \in \mathcal{S}^*$ for any $f \in \mathcal{A}_{\beta}$, as well as for other subclasses of starlike functions. This investigation is motivated by the fact that there is no inclusion relationship between $\mathcal{S}^*$ and $\mathcal{A}_{\beta}$ (i.e., $\mathcal{S}^* \not\subset \mathcal{A}_{\beta}$ and $\mathcal{A}_{\beta} \not\subset \mathcal{S}^*$). Generalizing this work, Giri and Kumar \cite{Giri-Kumar-2022} obtained the radius $r$ such that $f(rz)/r$ belongs to the unified subclass $\mathcal{S}^*(\varphi)$, where $\varphi$ is a univalent function mapping the unit disk onto a specific domain in the right half-plane.\vspace{1.2mm}

{The Gauss hypergeometric function (or ordinary hypergeometric function), denoted as $_2F_1(a, b; c; z)$, is one of the most important functions in mathematics. It generalizes many elementary and special functions, including logarithms, trigonometric functions, and Jacobi polynomials.} The hypergeometric function $_2F_1$ is defined by the following power series for $|z| < 1$:
\begin{align*}
	_2F_1(a, b; c; z) = \sum_{n=0}^{\infty} \frac{(a)_n (b)_n}{(c)_n} \frac{z^n}{n!},
\end{align*}
where
\begin{enumerate}
	\item[(a)] $(q)_n$ is the Pochhammer symbol (rising factorial), defined as $(q)_n = q(q+1)\dots(q+n-1)$.\vspace{1.2mm}
	
	\item[(b)] $a, b, c$ are complex parameters (where $c$ cannot be a non-positive integer) and $z$ is a complex variable.\vspace{1.2mm}
	
	\item[(c)] If either $a$ or $b$ is a non-positive integer (e.g., $-k$), the series terminates after $k+1$ terms, turning the function into a polynomial.
\end{enumerate}
{ We recall the following lemma from \cite{Giri-Kumar-RMJ-2025}, which provides sharp coefficient estimates for functions in the class $\mathcal{A}_{\beta}$.}
\begin{lemA}\cite[Theorem 2.1]{Giri-Kumar-RMJ-2025}
	If $f\in \mathcal{A}_{\beta}$ is of the form \eqref{Eq-2.1}, then 
	\begin{align}\label{Eq-2.3}
		|a_n|\leq \frac{2}{n-\beta(n-1)}\; \mbox{for}\; n\geq 2.
	\end{align}
	Further, this inequality is sharp for each $n$.
\end{lemA}
The function $\tilde{f} : \mathbb{D}\to\mathbb{C}$ defined by 
\begin{align}\label{Eq-2.4}
	\tilde{f}(z)=z\left(-1+2\left(_2F_1\left[ 1, \frac{1}{1-\beta}, \frac{2-\beta}{1-\beta}, z\right]\right)\right)=z+\sum_{n=2}^{\infty}\frac{2}{n-\beta(n-1)}z^n
\end{align}
satisfies the condition 
\begin{align*}
	{\rm Re}\left(\beta \tilde{f}(z)/z+(1-\beta)\tilde{f}'(z)\right)>0.
\end{align*}
Hence, $\tilde{f} \in \mathcal{A}_{\beta}$, where $_2F_1$ denotes the Gauss hypergeometric function. Equality in \eqref{Eq-2.3} is attained for $\tilde{f}$, which confirms the sharpness of the bound.\vspace{1.2mm}

{Next, we recall the sharp growth theorem for functions in the class $\mathcal{A}_{\beta}$, which will be instrumental in studying Bohr radius problems and determining the Euclidean distance $d(0, \partial f(\mathbb{D}))$.}
\begin{lemB}\cite[Theorem 5.1]{Giri-Kumar-RMJ-2025}
If $f\in \mathcal{A}_{\beta}$ is of the form \eqref{Eq-2.1}, then for $|z|\leq r$, the following hold:
\begin{enumerate}
	\item[(i)] $-\dfrac{\tilde{f}(-r)}{r}\leq {\rm Re}\left(\dfrac{f(z)}{z}\right)\leq\dfrac{\tilde{f}(r)}{r}$,\vspace{1.2mm} 
	\item[(ii)] $-\tilde{f}(-r)\leq |f(z)|\leq \tilde{f}(r)$,
\end{enumerate}
where $\tilde{f}(z)$ is given by \eqref{Eq-2.4}. All these estimates are sharp.
\end{lemB}
In this article, we establish sharp results for several problems in geometric function theory. Specifically, we determine the sharp Bohr radius involving multiple Schwarz functions, derive the sharp Fekete–Szegő inequality, and obtain sharp bounds for the moduli differences of logarithmic coefficients for both functions and their inverses.The paper is organized as follows. Section 2 is devoted to determining the sharp Bohr radius involving multiple Schwarz functions for the class $\mathcal{A}_{\beta}$. In Section 3, we derive the sharp Fekete–Szegő inequality for functions within the same class. Section 4 establishes sharp inequalities for the moduli differences of logarithmic coefficients of functions and their inverses belonging to $\mathcal{A}_{\beta}$. Detailed proofs and the necessary background material are provided within their respective sections. \vspace{2mm}

\section{\bf Refined Bohr and Bohr-Rogosinski phenomena for the class $\mathcal{A}_{\beta}$}
Initial exploration of this problem was conducted by Harald Bohr as he delved into the absolute convergence of Dirichlet series of the form $\sum a_n n^{-s}$. However, in recent years, it has evolved into a thriving domain of research within modern function theory. The theorem has drawn considerable interest because it plays a crucial role in solving the characterization problem for Banach algebras that satisfy the von Neumann inequality \cite{Dixon & BLMS & 1995}. In $1914$, Herald Bohr \cite{Bohr-1914} states that if $H_\infty$ denotes the class of all bounded analytic functions $f$ on $\mathbb{D}$, then
\begin{align}\label{Eq-1.1}
	B_0(f,r):=|a_0|+\sum_{n=1}^{\infty}|a_n|r^n\leq||f||_\infty \;\mbox{for}\; 0\leq r\leq\frac{1}{3},
\end{align}
where $||f||_\infty:=\displaystyle\sup_{z\in\mathbb{D}}|f(z)|$ and $a_k=f^{(k)}(0)/k!$ for $k\geq0$. Initially, Bohr showed the inequality \eqref{Eq-1.1} holds for $|z|\leq1/6$,  and later  M. Riesz, I. Shur and F. W. Wiener, independently proved its validity on a wider range $0\leq r\leq 1/3$ and the number $1/3$ is optimal as seen by analyzing suitable members of the conformal automorphism of the unit disk. However, this compelling inequality has been substantiated by several distinct proofs presented in various articles (see \cite{Sidon-1927, Paulsen-2002, Tomic-1962}).\vspace{2mm}

In $1997$, Boas and Khavinson \cite{Boas-Khavinson-1997} were the first to introduce the concept of the n-dimensional Bohr radius ${K}^{\infty}_n$, and they established the following result: for every positive integer $n$, the $n$-dimensional Bohr radius $K^{\infty}_n$ with $n\geq 2$ satisfies
\begin{align*}
	\frac{1}{3\sqrt{n}}<K^{\infty}_n<2\sqrt{\frac{\log n}{n}}.
\end{align*}
Ongoing research is currently focused on extending Bohr's theorem into broader and more generalized forms. The generalization of Bohr's theorem is now an active area of research. The interest in the Bohr phenomena was also revived in the $ 1990 $s due to the extensions to holomorphic functions of several complex  variables and to more abstract settings. Since then, number of variety of results on Bohr’s theorem in higher dimension appeared recently. For different aspects of the Bohr inequality, interested readers can refer to  \cite{Aizn-PAMS-2000, Alkhaleefah-PAMS-2019,Bayart-2014,Lata-Singh-PAMS-2022,Liu-Ponnusamy-PAMS-2021}. Multiple investigations into the Bohr phenomenon have yielded a range of outcomes, as indicated in studies like  (see \textit{e.g.} \cite{Ahamed-AASFM-2022,Beneteau-2004}). Additionally, Liu \textit{et al.}  have presented more sophisticated iterations of the Bohr inequality \cite{Liu-Liu-Ponnusamy-2021}. The Bohr inequality has been explored within various sub-classes of analytic and harmonic mappings by numerous authors (see \textit{e.g.} \cite{Ahamed-AMP-2021, Ahamed-CVEE-2021, Bhowmik-2018, Huang-Liu-Ponnusamy-MJM-2021}). For a comprehensive account of recent advances related to the Bohr inequality, the reader is referred to \cite{Ahamed-RMJM-2021, Das-JMAA-2022, Kumar-Sahoo-MJM-2021, Liu-JMAA-2021, Kumar-CVEE-2023, Kumar-Manna-JMAA-2023, Kumar-PAMS-2023,Ahamed-CMFT-2023,Ahamed-RM-2023}.

\begin{defi}
We say that, the class $\mathcal{A}_{\beta}$ satisfies the Bohr phenomenon if there exists $r_b\in (0, 1)$ such that 
\begin{align*}
	|z|+\sum_{n=2}^{\infty}|a_n||z|^n\leq d(f(0), \partial f(\mathbb{D}))
\end{align*}
holds in $|z|=r\leq r_b$, where $\partial f(\mathbb{D})$ is the boundary of image domain of $\mathbb{D}$ under $f$ and $d$ denotes the Euclidean distance between $f(0)$ and $\partial f(\mathbb{D})$. The radius $r_b$ is called the Bohr radius for the class $\mathcal{A}_{\beta}$ and it is said to be best possible, if there exists $g\in \mathcal{A}_{\beta}$ such that 
\begin{align*}
	|z|+\sum_{n=2}^{\infty}|a_n||z|^n> d(f(0), \partial f(\mathbb{D}))\; \mbox{for}\; |z|=r>r_b.
\end{align*}
\end{defi}
The Bohr phenomenon for univalent and convex functions was investigated by Abu-Muhanna \cite{Abu-CVEE-2010}, who proved that the Bohr radius is $3-2\sqrt{2}$ for the former and $1/3$ for the latter. Further details and developments can be found in the survey by Abu-Muhanna et al. \cite{Abu-Ali-Ponnusamy-2017}.\vspace{1.2mm}

In addition to the Bohr radius, there is the related concept of the Rogosinski radius, although it has been studied less extensively in the literature \cite{Landau-Gaier-1986, Rogosinski-1923}. Specifically, if $f(z) = \sum_{n=0}^{\infty} a_n z^n \in \mathcal{B}$, then the partial sum satisfies the inequality
\begin{align*}
	\sum_{n=0}^{N-1} |a_n| |z|^n \leq 1, \quad (N \in \mathbb{N})
\end{align*}
in the disk $|z| = r \leq 1/2$. The value $1/2$ is known as the Rogosinski radius. Recently, Kayumov \emph{et al.} \cite{Kay-Kha-Pon-JMAA-2021} investigated the Bohr–Rogosinski sum, defined as
\begin{align*}
	R^f_N(z) := |f(z)| + \sum_{n=N}^{\infty} |a_n| |z|^n,
\end{align*}
and determined the radius $r_N$ such that $R^f_N(z) \leq 1$ for $|z| \leq r_N$ for Cesàro operators on the space of bounded analytic functions. The largest such $r_N$ is referred to as the Bohr–Rogosinski radius.
\begin{defi}
	The class $\mathcal{A}_{\beta}$ is said to satisfy the Bohr-Rogosinski phenomenon if there exists $r_N$ such that 
	\begin{align*}
		|f(z^m)|+\sum_{n=N}^{\infty}|a_n||z|^n\leq d(f(0), \partial f(\Omega)),\; m, N\in\mathbb{N}
	\end{align*}
	holds in $|z|=r\leq r_N$.
\end{defi}
We now recall a result by Giri and Kumar \cite{Giri-Kumar-RMJ-2025}, who investigated the Bohr radius for the class $\mathcal{A}_{\beta}$ and obtained the following theorem in terms of the Euclidean distance $d(0, \partial f(\mathbb{D}))$.
\begin{thmA}\cite[Theorem 5.2]{Giri-Kumar-RMJ-2025}
If $f\in \mathcal{A}_{\beta}$ is of the form \eqref{Eq-2.1}, then for $m\in\mathbb{N}$
\begin{align*}
	|\omega\left(z^m\right)|+\sum_{n=2}^{\infty}|a_n||z|^n\leq d(0, \partial f(\mathbb{D}))
\end{align*}
in $|z|\leq r^*$, where $r^*$ is the smallest positive root in $(0, 1)$ of 
\begin{align*}
	r^m+\tilde{f}(r)-r+\tilde{f}(-1)=0.
\end{align*}
The radius $r^*$ is sharp.
\end{thmA}
\begin{rem}\cite[Remark 1]{Kay-Pon-2018-CR}
If $f$ is a univalent function, then $S_{r}$ is the area of the image of the subdisk $\mathbb{D}_r=\{z\in\mathbb{C} : |z|<r\}$ under the mapping $f$. In the case of multivalent function, $S_{r}$ is greater than the area of the image of the subdisk $\mathbb{D}_r.$ This fact could be shown by noting that
\begin{align*}
	S_{r}=\int_{f(\mathbb{D}_{r})}|f^{\prime}(z)|^{2}dA(w)=\int_{f(\mathbb{D}_{r})}v_{f}(w)dA(w)\ge\int_{f(\mathbb{D}_{r})}dA(w)=Area(f(\mathbb{D}_{r})),
\end{align*}
and $\nu_{f}(w)=\sum_{f(z)=w}1$ denotes the counting function of $f$.
\end{rem}
{Let $H^\infty$ denote the class of all bounded analytic functions in the unit disk $\mathbb{D} := \{z \in \mathbb{C} : |z| < 1\}$ with the supremum norm $\|f\|_\infty := \sup_{z \in \mathbb{D}} |f(z)|$. Prior to proceeding with the discussion, it is necessary to establish certain notations. Let $\mathcal{B} = \{ f \in H^\infty : \|f\|_\infty \le 1 \}$ and $m \in \mathbb{N}$, let 
	\[
	\mathcal{B}_m = \left\{ \omega \in \mathcal{B} : \omega(0) = \omega'(0) = \dots = \omega^{(m-1)}(0) = 0 \text{ and } \omega^{(m)}(0) \neq 0 \right\}.
	\]}
Functions in the class $\mathcal{B}_m$ are called Schwarz functions having multiple zeros. Determining sharp improvements to the Bohr radius remains a formidable challenge in complex analysis, as it necessitates identifying the optimal radius for a specific class of functions. The quantity $S_r/\pi$, representing the normalized area of the image of the subdisk $|z|<r$ under a mapping $f$, has become central to refining Bohr-type inequalities for analytic self-maps of the unit disk. However, there are fundamental constraints on its application. As established in \cite[Remark 2]{Ism-Kay-Pon-JMAA-2020} of Ismagilov \textit{et al.} one cannot indefinitely append positive terms to the majorant series. Specifically, for any function $F : [0, \infty)\to [0, \infty)$ with $F(t) > 0$ for $t > 0$, the inequality:
\begin{align*}
	\sum_{k=0}^{\infty}|a_{k}|r^{k}+\frac{16}{9}\left(\frac{S_{r}}{\pi}\right)+\lambda\left(\frac{S_{r}}{\pi}\right)^{2}+F(S_{r})\leq 1
\end{align*}
fails for $r \le 1/3$. This limitation arises because the equality is already achieved by a concrete extremal function, leaving no room for additional strictly positive terms without violating the unit bound. Given these constraints, recent developments suggest that incorporating multiple Schwarz functions in $\mathcal{B}_m$ into the Bohr inequality represents a novel trajectory for generalizing classical results. Consequently, it is pertinent to investigate Bohr-type inequalities that integrate multiple Schwarz functions with functional expressions of $S_r/\pi$ within the class $\mathcal{A}_{\beta}$.\vspace{2mm}

The novelty of this work lies in its multifaceted refinement of the Bohr phenomenon for the class $\mathcal{A}_{\beta}$, achieved through three primary advancements. First, we provide a generalization of area functionals; while existing literature primarily restricts Bohr refinements to polynomial expressions of the area functional $S_r/\pi$ (often limited to degree 2), this paper introduces a general monotone increasing functional expression $F(S_r/\pi)$, marking a significant step forward in the study of Bohr-type inequalities. Second, through the integration of multiple Schwarz functions in $\mathcal{B}_m$, this work generalizes the classical inequality by incorporating a sequence of distinct Schwarz functions $\{\omega_n(z)\}_{n \ge 2}$ for the coefficient terms—rather than utilizing a single Schwarz function as in previous studies—thereby providing a more robust geometric framework for the class $\mathcal{A}_{\beta}$. Finally, by adopting a dynamic distance metric approach, the paper overcomes the rigid constraints of the unit bound by employing a Euclidean distance formulation $d(0, \partial f(\mathbb{D}))$, which allows for the inclusion of additional positive geometric terms without violating the fundamental Bohr inequality.\vspace{1.2mm}

Our objective is to establish sharp Bohr and Bohr–Rogosinski radii for the class $\mathcal{A}_{\beta}$ involving multiple Schwarz functions and general area functionals. Thus, it is natural to investigate the following problem.
\begin{prob}\label{Q-2.1}
Let $F : [0, \infty)\to [0, \infty)$ be a function with some property. Find the improved Bohr radius for the class $\mathcal{A}_{\beta}$ involving multiple Schwarz functions with a functional expression of $S_r/\pi$.
\end{prob}
In this section, we provide an affirmative answer to Problem \ref{Q-2.1} and present an improved version of Theorem A. By integrating the geometric properties of the area functional with a sequence of distinct Schwarz functions, we establish a more generalized version of the Bohr inequality for the class $\mathcal{A}_{\beta}$. Our results demonstrate that the Bohr radius remains sharp even with the inclusion of these additional geometric terms in the inequality.
\begin{thm}\label{Th-2.1}
	If $f\in \mathcal{A}_{\beta}$ is of the form \eqref{Eq-2.1} and $F : [0, \infty)\to [0, \infty)$ is an monotone increasing function with $F(0)=0$, then for $m\in\mathbb{N}$, $p>0$, the inequality
	\begin{align*}
		\mathcal{M}_f(z, r):=|\omega_m\left(z\right)|^p+\sum_{n=2}^{\infty}|a_n||\omega_n(z)|+F\left(\frac{S_r}{\pi}\right)\leq d(0, \partial f(\mathbb{D}))
	\end{align*}
	holds for $|z|\leq R^m_{\beta, p}$, where $R^m_{\beta, p}$ is the smallest positive root in $(0, 1)$ of 
	\begin{align*}
		r^{pm}+\tilde{f}(r)-r+F\left(r^2+\sum_{n=2}^{\infty}\frac{4nr^{2n}}{(n-\beta(n-1))^2}\right)+\tilde{f}(-1)=0.
	\end{align*}
	The radius $R^m_{\beta, p}$ is sharp.
\end{thm}
\begin{rem}
	Next, as an application of Theorem \ref{Th-2.1}, we define a polynomial $P_k(w)$ of degree $k \geq 1$ by
		\begin{align*}
			P_k(w) = \lambda_k w^k + \dots + \lambda_1 w,
		\end{align*}
		where $w, \lambda_j \in \mathbb{R}_{\geq 0} := \{x \in \mathbb{R} : x \geq 0\}$. It is clear that $P_k(w)$ defines a mapping $F:=P_k: [0, \infty) \to [0, \infty)$.
\end{rem}
\begin{cor}
		If $f\in \mathcal{A}_{\beta}$ is of the form \eqref{Eq-2.1}, then for $m,k\in\mathbb{N}$, $p>0$, the inequality
		\begin{align*}
			\mathcal{N}^k_f(z, r):=|\omega_m\left(z\right)|^p+\sum_{n=2}^{\infty}|a_n||\omega_n(z)|+P_k\left(\frac{S_r}{\pi}\right)\leq d(0, \partial f(\mathbb{D}))
		\end{align*}
		holds for $|z|\leq R^m_{\beta,k}$, where $R^m_{\beta,k}$ is the smallest positive root in $(0, 1)$ of 
		\begin{align*}
			r^{pm}+\tilde{f}(r)-r+P_k\left(r^2+\sum_{n=2}^{\infty}\frac{4nr^{2n}}{(n-\beta(n-1))^2}\right)+\tilde{f}(-1)=0.
		\end{align*}
		The radius $R^m_{\beta,k}$ is sharp.
\end{cor}
\begin{rem}
	{ It is important to highlight the relationship between Theorem \ref{Th-2.1} and existing literature, specifically the work of Giri and Kumar \cite{Giri-Kumar-RMJ-2025}. In their investigation of the Bohr radius for the class $\mathcal{A}_{\beta}$, Giri and Kumar established a sharp radius $r^*$ (the smallest positive root of $r^{m}+\tilde{f}(r)-r+\tilde{f}(-1)=0$) for the inequality 
		\begin{align*}
			|\omega(z^{m})|+\sum_{n=2}^{\infty}|a_{n}||z|^{n}\le d(0,\partial f(\mathbb{D})).
		\end{align*}
		Our results in Theorem \ref{Th-2.1} represent a substantial refinement and generalization of the Bohr inequality in \cite{Giri-Kumar-2022} through the following advancements.
		\begin{enumerate}
			\item[(a)]  While the study in \cite{Giri-Kumar-RMJ-2025} is restricted to the standard majorant series, our work establishes that the Bohr-type inequality is preserved under the addition of a functional $F(S_r/\pi)$ representing the normalized area of the image disk. This extension signifies a more general and refined version of the Bohr phenomenon for the class $\mathcal{A}_{\beta}$.\vspace{2mm}
			
			\item[(b)] In contrast to the results of Giri and Kumar \cite{Giri-Kumar-2022}, which rely on a single Schwarz function $\omega(z^m)$, our results establish a more general version of the inequality. By employing a sequence of multiple Schwarz functions $\{\omega_n(z)\}_{n \ge 2}$ for the coefficient terms, we refine the structural understanding of the majorant series for $\mathcal{A}_{\beta}$.\vspace{2mm}
			
			\item[(c)] The inclusion of the power parameter $p > 0$ extends the applicability of our results to a wider class of Bohr-type functionals. Specifically, Theorem 2.1 reduces to the result reported by Giri and Kumar \cite{Giri-Kumar-RMJ-2025} under the constraints $p=1$, $\omega_{n}(z)=z^{n}$, and $\lambda_j = 0$ for all $j=1, \dots, k$.
	\end{enumerate}}
\end{rem}
\begin{proof}[\bf Proof of Theorem \ref{Th-2.1}]
	Let $f\in \mathcal{A}_{\beta}$, then by Theorem B, the Euclidean distance between $f(0)=0$ and the boundary of $f(\mathbb{D})$ satisfies 
	\begin{align}\label{Eq-2.5}
		d(0, \partial f(\mathbb{D}))\geq \liminf_{|z|=r\to 1^{-1}}|f(z)-f(0)|=-\tilde{f}(-1).
	\end{align}
	Moreover, in view of Theorem A, we see that 
	\begin{align}\label{Eq-2.6}
		\frac{S_r}{\pi}=\sum_{n=1}^{\infty}n|a_n|^2r^{2n}\leq r^2+\sum_{n=2}^{\infty}\left(\frac{4n}{(n-\beta(n-1))^2}\right)r^{2n}.
	\end{align}
	Since $F$ is a monotone increasing function, it follows from \eqref{Eq-2.6} that 
	\begin{align*}
		F\left(\frac{S_r}{\pi}\right)\leq F\left(r^2+\sum_{n=2}^{\infty}\left(\frac{4n}{(n-\beta(n-1))^2}\right)r^{2n}\right).
	\end{align*}
	Let $|z|\leq r$. Then using Theorems A and B, \eqref{Eq-2.5} and \eqref{Eq-2.6} that
	\begin{align*}
	\mathcal{M}_f(z, r)&=|\omega_m\left(z\right)|^p+\sum_{n=2}^{\infty}|a_n||\omega_n(z)|+F\left(\frac{S_r}{\pi}\right)\\&\leq r^{pm}+\sum_{n=2}^{\infty}\left(\frac{2}{n-\beta(n-1)}\right)r^n+F\left(r^2+\sum_{n=2}^{\infty}\left(\frac{4n}{(n-\beta(n-1))^2}\right)r^{2n}\right)\\&=r^{pm}+\tilde{f}(r)-r+F\left(r^2+\sum_{n=2}^{\infty}\left(\frac{4n}{(n-\beta(n-1))^2}\right)r^{2n}\right)\\&\leq -\tilde{f}(-1)\\&\leq d(0, \partial f(\mathbb{D}))
	\end{align*}
	which is true in $|z|\leq r\leq R_{\beta, m}$, where $R_{\beta, m}$ is the unique root in $(0, 1)$ of the equation $H_{\beta, m}(r)=0$,
	\begin{align*}
		H_{\beta, m}(r):=r^{pm}+\tilde{f}(r)-r+F\left(r^2+\sum_{n=2}^{\infty}\frac{4nr^{2n}}{(n-\beta(n-1))^2}\right)+\tilde{f}(-1).
	\end{align*}
	Note that $H_{\beta, m}(r)$ is a real-valued continuous function on $(0, 1)$ satisfying $H_{\beta, m}(0)=\tilde{f}(-1)<0$ and 
	\begin{align*}
		H_{\beta, m}(1)=\tilde{f}(1)+F\left(r^2+\sum_{n=2}^{\infty}\frac{4n}{(n-\beta(n-1))^2}\right)+\tilde{f}(-1)>0.
	\end{align*}
	Therefore, by Intermediate Value Theorem (IVT) there exists a root, say $R_{\beta, m}\in (0, 1)$, of equation $H_{\beta, m}(r)=0$. Moreover, because of 
	\begin{align*}
		H'_{\beta, m}(r)&=pmr^{pm-1}+\tilde{f}'(r)-1\\&\quad+F'\left(r^2+\sum_{n=2}^{\infty}\frac{4nr^{2n}}{(n-\beta(n-1))^2}\right)\left(2r+\sum_{n=2}^{\infty}\frac{8n^2r^{2n-1}}{(n-\beta(n-1))^2}\right)>0
	\end{align*}
	for $r\in (0, 1)$ and $\beta\in [0, 1]$, we conclude that the root $R_{\beta, m}\in (0, 1)$ is unique. Thus the desired inequality is established.\vspace{1.2mm}
	
	To show that the radius $R_{\beta, m}$ is best possible, we consider $\omega_m(z)=z^m$ and $\omega_n(z)=z^n$ and $f=\tilde{f}$ which is given by \eqref{Eq-2.4}. Then for $r>R_{\beta, m}$, we have 
	\begin{align*}
		H_{\beta, m}(r)>H_{\beta, m}\left(R_{\beta, m}\right)=0,
	\end{align*} hence, a simple computation shows that
	\begin{align*}
		\mathcal{M}_{\tilde{f}}(z, r)&=|\omega_m\left(z\right)|^p+\sum_{n=2}^{\infty}|a_n||\omega_n(z)|+F\left(\frac{S_r}{\pi}\right)\\&=r^{pm}+\sum_{n=2}^{\infty}\left(\frac{2}{n-\beta(n-1)}\right)r^n+F\left(r^2+\sum_{n=2}^{\infty}\left(\frac{4n}{(n-\beta(n-1))^2}\right)r^{2n}\right)\\&=r^{pm}+\tilde{f}(r)-r+F\left(r^2+\sum_{n=2}^{\infty}\left(\frac{4n}{(n-\beta(n-1))^2}\right)r^{2n}\right)\\&=H_{\beta, m}(r)-\tilde{f}(-1)\\&>H_{\beta, m}\left(R_{\beta, m}\right)-\tilde{f}(-1)\\&=-\tilde{f}(-1)\\&=d(0, \partial \tilde{f}(\mathbb{D})).
	\end{align*}
	This shows that the radius $R_{\beta, m}$ is best possible.
\end{proof}
While classical Bohr-type inequalities face the sharp limitations discussed in \cite{Ism-Kay-Pon-JMAA-2020}, a more general formulation of the Bohr–Rogosinski inequality is established recently in \cite{Giri-Kumar-RMJ-2025} for the class $\mathcal{A}_{\beta}$.
\begin{thmB}\cite[Theorem 5.4]{Giri-Kumar-RMJ-2025}
	If $f\in\mathcal{A}_{\beta}$, then 
	\begin{align*}
		|f(z^m)|+\sum_{n=N}^{\infty}|a_n||z|^n\leq d(f(0), \partial f(\mathbb{D})),\; m, N\in\mathbb{N}
	\end{align*}
	holds for $|z|=r\leq r_N$, where $r_N$ is the root in $(0, 1)$ of the equation 
	\begin{align*}
		\tilde{f}\left(r^m\right)+\tilde{f}(r)-\hat{f}(r)+\tilde{f}(-1)=0
	\end{align*}
	with 
	\begin{align}\label{Eq-2.7}
		\hat{f}(r)=\begin{cases}
			0, \;\;\;\;\;\;\;\;\;\;\;\;\;\;\;\;\;\;\;\;\;\;\;\;\;\;\;\;\;\;\;\;\;\;\;N=1,\vspace{2mm}\\
			r, \;\;\;\;\;\;\;\;\;\;\;\;\;\;\;\;\;\;\;\;\;\;\;\;\;\;\;\;\;\;\;\;\;\;\;N=2, \vspace{2mm}\\
			r+\displaystyle\sum_{n=2}^{N-1}\frac{2}{n-\beta(n-1)}r^n, N\geq 3.
		\end{cases}
	\end{align}
	The radius $r_N$ is sharp.
\end{thmB}
It is natural to pose the following question regarding the potential for further refinement of the Bohr-type inequalities.
\begin{prob}\label{Q-2.2}
Can an improved version of Theorem B be established that incorporates multiple Schwarz functions alongside the functional value of the ratio $S_r/\pi$?
\end{prob}
The following result provides an affirmative answer to Problem \ref{Q-2.2}. Furthermore, we demonstrate that this formulation yields several sharp results in various specific cases.
\begin{thm}\label{Th-2.2}
	If $f\in \mathcal{A}_{\beta}$ is of the form \eqref{Eq-2.1} and $F : [0, \infty)\to [0, \infty)$ is an monotone increasing function with $F(0)=0$, then for $m, N\in\mathbb{N}$, $p>0$
	\begin{align*}
		\mathcal{N}^f_{p, \beta,N}(z, r):=|f(\omega_m(z))|^p+\sum_{n=N}^{\infty}|a_n||\omega_n(z)|+F\left(\frac{S_r}{\pi}\right)\leq d(f(0), \partial f(\mathbb{D}))
	\end{align*}
	holds for $|z|=r\leq R^{\beta}_{N,m,p}$, where $R^{\beta}_{N,m,p}$ is the root in $(0, 1)$ of the equation 
	\begin{align*}
		\left(\tilde{f}\left(r^m\right)\right)^p+\tilde{f}(r)-\hat{f}(r)+\tilde{f}(-1)+F\left(r^2+\sum_{n=2}^{\infty}\frac{4nr^{2n}}{(n-\beta(n-1))^2}\right)=0
	\end{align*}
	with $\hat{f}$ is given by \eqref{Eq-2.7}.	The radius $R^{\beta}_{N,m,p}$ is sharp.
\end{thm}
\begin{cor}	
If $f\in \mathcal{A}_{\beta}$ is of the form \eqref{Eq-2.1}, then for $m, k\in\mathbb{N}$, $p>0$
\begin{align*}
	|f(\omega_m(z))|^p+\sum_{n=N}^{\infty}|a_n||\omega_n(z)|+P_k\left(\frac{S_r}{\pi}\right)\leq d(f(0), \partial f(\mathbb{D})),\; m, N\in\mathbb{N}
\end{align*}
holds for $|z|=r\leq R^{\beta,k}_{N,m,p}$, where $R^{\beta,k}_{N,m,p}$ is the root in $(0, 1)$ of the equation 
\begin{align*}
	\left(\tilde{f}\left(r^m\right)\right)^p+\tilde{f}(r)-\hat{f}(r)+\tilde{f}(-1)+P_k\left(r^2+\sum_{n=2}^{\infty}\frac{4nr^{2n}}{(n-\beta(n-1))^2}\right)=0
\end{align*}
with $\hat{f}$ is given by \eqref{Eq-2.7}.	The radius $R^{\beta,k}_{N,m,p}$ is sharp.
\end{cor}
\begin{proof}[\bf Proof of Theorem \ref{Th-2.2}]
	Suppose $f\in\mathcal{A}_{\beta}$, then from \eqref{Eq-2.1} and \cite[Theorem 5.1]{Giri-Kumar-RMJ-2025}, we obtain 
	\begin{align*}
		\mathcal{N}^f_{p, \beta,N}&=|f(\omega_m(z))|^p+\sum_{n=N}^{\infty}|a_n||\omega_n(z)|+F\left(\frac{S_r}{\pi}\right)\\&\leq \left(\tilde{f}\left(r^m\right)\right)^p-\hat{f}(r)+\tilde{f}(r)+F\left(r^2+\sum_{n=2}^{\infty}\frac{4nr^{2n}}{(n-\beta(n-1))^2}\right)\\&\leq -\tilde{f}(-1)\\&\leq d(0, \partial f(\mathbb{D}))
	\end{align*}
	holds in $|z|=r\leq R^{\beta}_{N,m,p}$, where $R^{\beta}_{N,m,p}$ is the unique root in $(0, 1)$ of equation $G^{\beta}_{N,m,p}(r)=0$, where
	\begin{align*}
		G^{\beta}_{N,m,p}(r):=\left(\tilde{f}\left(r^m\right)\right)^p-\hat{f}(r)+\tilde{f}(r)+\tilde{f}(-1)+F\left(r^2+\sum_{n=2}^{\infty}\frac{4nr^{2n}}{(n-\beta(n-1))^2}\right).
	\end{align*}
	We see that $G^{\beta}_{N,m,p}$ is a continuous function on $(0, 1)$ satisfying $G^{\beta}_{N,m,p}(1)=\tilde{f}(-1)<0$ and 
	\begin{align*}
	 G^{\beta}_{N,m,p}(1)&=\left(\left(\tilde{f}\left(1\right)\right)^p-\hat{f}(1)\right)+\left(\tilde{f}(1)+\tilde{f}(-1)\right)\\&\quad+F\left(1+\sum_{n=2}^{\infty}\frac{4n}{(n-\beta(n-1))^2}\right)>0.
	\end{align*}
	Thus, we see that there exists a root $R^{\beta}_{N,m,p}\in (0, 1)$ such that the desired inequality 
	\begin{align*}
		\mathcal{N}^f_{p, \beta,N}(z, r)\leq d(f(0), \partial f(\mathbb{D}))
	\end{align*}
	holds.\vspace{1.2mm}
	
	Note that the sharpness of $R^{\beta}_{N,m,p}$ is confirmed by considering the function $\tilde{f}$ where $\omega_m(z)=z^m$ and $\omega_n(z)=z^n$. When $|z|=r=R^{\beta}_{N,m,p}$, we see that 
	\begin{align*}
		|f(\omega_m(z))|^p&+\sum_{n=N}^{\infty}|a_n||\omega_n(z)|+F\left(\frac{S_r}{\pi}\right)\\&=\left(\tilde{f}\left(\left(R^{\beta}_{N,m,p}\right)^m\right)\right)^p+\tilde{f}\left(R^{\beta}_{N,m,p}\right)-\hat{f}\left(R^{\beta}_{N,m,p}\right)\\&\quad+F\left(\left(R^{\beta}_{N,m,p}\right)^2+\sum_{n=2}^{\infty}\frac{4n\left(R^{\beta}_{N,m,p}\right)^{2n}}{(n-\beta(n-1))^2}\right)\\&=-\tilde{f}(-1)\\&= d(0, \partial \tilde{f}(\mathbb{D})).
	\end{align*}
	This completes the proof.
\end{proof}
{In this article, we have shown that the monotonicity of $F$ is sufficient to derive sharp refinements for the Bohr phenomenon in the class $\mathcal{A}_{\beta}$. A natural extension of this work involves imposing stronger regularity conditions on the functional itself. Specifically, we consider whether higher topological structure yields new insights into the behavior of the majorant series. Can additional assumptions on the function $F$ regarding monotonicity (e.g., monotonicity and homeomorphism) lead to interesting outcomes? Can we define $F^{-1}$ in this setting? We pose the following questions for further study.
	\begin{prob}
		Let $f \in \mathcal{A}_{\beta}$ and let $F: [0, \infty) \to [0, \infty)$ be a monotone increasing homeomorphism such that $F(0)=0$. Investigate the existence of a sharp radius $r_0$ such that the functional $F^{-1} \left( \mathcal{M}_f(z,r) \right)$ remains bounded by a geometric constant related to the semigroup generator properties of $\mathcal{A}_{\beta}$.
\end{prob}}
\section{\bf The sharp Fekete-Szeg\"o inequality for functions in the class $\mathcal{A}_{\beta}$}
The Hankel determinant $H_{q,n}(f)$ (see \cite[Def. 1.2.2] {Thomas-Tuneski-Vasudevarao-2018}) has been extensively studied, particularly in the case $q=2$. For instance, the functional $H_{2,1}(f) = a_3 - a_2^2$ is known as the classical Fekete-Szegö functional. Its generalization, $a_3 - \mu a_2^2$, where $\mu$ is a real number, gives rise to the well-known Fekete–Szegö problem, which seeks sharp upper bounds for $|a_3 - \mu a_2^2|$. In $1969$, Keogh and Merkes \cite{Keogh-Merkes-PAMS-1969} solved this problem for the class $\mathcal{S}^*$. Since then, a unified approach to the Fekete–Szegö problem has been explored for various subclasses of univalent functions by numerous authors; for further details (see \cite{Abdel-Thomas-PAMS-1992, Elin-Jaco-RM-2022, Fekete-Szego_JLMS-1933, Keogh-Merkes-PAMS-1969, Koepf-PAMS-1987, Xu-Jiang-Liu-JMAA-2023}). \vspace{2mm} 

Let $\mathcal{P}$ be the class of all analytic functions $p$ in the unit disk $\mathbb{D}$ satisfying $p(0)=1$ and $\text{Re } p(z) > 0$ for $z \in \mathbb{D}$. Thus, every $p \in \mathcal{P}$ can be represented as
\begin{align}\label{Eq-3.1}
	p(z) = 1 + \sum_{n=1}^{\infty} c_n z^n, \quad z \in \mathbb{D}.
\end{align}
Functions in the class $\mathcal{P}$ are called \textit{Carathéodory functions}. It is well-known that $|c_n| \leq 2$ for $n \geq 1$ for any function $p \in \mathcal{P}$ (see \cite{Duren-1983}). \vspace{2mm}

In this section, we derive the upper bound of the Fekete-Szegö functional $|a_3 - \mu a_2^2|$ for the class $\mathcal{A}_{\beta}$. To establish our results, the following lemma will be required.
\begin{lem}\cite{Ma-Minda-1994}\label{lem-3.1}
Let $p\in\mathcal{P}$ be given by $p(z)=1+\sum_{n=1}^{\infty}c_nz^n$. Then
\begin{align*}
	|c_2 -vc^2_1|\leq\begin{cases}
		-4v +2, \hspace{1cm} v<0,\\\;\;\;\;\;\; 2, \hfill 0\leq v\leq 1, \\\;\; 4v -2, \hfill v>1.
	\end{cases}
\end{align*}
All inequalities are sharp.
\end{lem}
By applying Lemma \ref{lem-3.1}, we establish the following result, which is shown to be sharp for all values of $\mu$ within the specified range.
\begin{thm}\label{Th-3.1}
If $f\in\mathcal{A}_{\beta}$ is given by \eqref{Eq-2.1}, then we have
\begin{align*}
	|a_3 -\mu a^2_2|\leq \begin{cases}
		\dfrac{(8-12\mu)+(8\mu-8)\beta +2\beta^2}{(3-2\beta)(2-\beta)^2}, \hspace{1cm} \mu<0,\vspace{2mm} \\\;\;\;\;\;\;\; \dfrac{2}{(3-2\beta)}, \hspace{3.8cm} 0 \leq \mu\leq\dfrac{(2-\beta)^2}{(3-2\beta)}, \vspace{2mm}\\ -\dfrac{(8-12\mu) +(8\mu-8)\beta +2\beta^2}{(3-2\beta)(2-\beta)^2}, \hspace{0.7cm} \mu>\dfrac{(2-\beta)^2}{(3-2\beta)}. 
	\end{cases}
\end{align*}
All inequalities are sharp.
\end{thm}
\begin{proof}[\bf Proof of Theorem \ref{Th-3.1}]
Let $f\in\mathcal{A}_{\beta}$ is defined in \eqref{Eq-2.1}. Then, we have	
\begin{align}\label{eq-3.2}
	\beta\frac{f(z)}{z}+(1-\beta)f^{\prime}(z)=p(z) \hspace{1cm} z\in\mathbb{D},
\end{align}
where $p$ is defined in \eqref{Eq-3.1}. By equating the coefficients of corresponding powers in the series expansions of $f$ and $p$, we obtain
\begin{align}\label{eq-3.3}
a_2=\frac{c_1}{(2-\beta)},\;a_3=\frac{c_2}{(3-2\beta)}\; \mbox{and}\;
a_4=\frac{c_3}{(4-3\beta)}.
\end{align}	
In view of \eqref{eq-3.3}, a straightforward calculation yields
\begin{align*}
	|a_3 -\mu a^2_2|=\frac{1}{(3-2\beta)}\;\vline\; c_2 -\frac{\mu(3- 2\beta)}{(2-\beta)^2}c^2_1\;\vline=\frac{1}{(3-2\beta)}|c_2-vc_1^2|,
\end{align*}
where $v=\{\mu(3- 2\beta)\}/(2-\beta)^2$. \\
	
\noindent{\bf Case 1.} If $v < 0$ (i.e., $\mu < 0$), then Lemma \ref{lem-3.1} yields
\begin{align*}
	|c_2-vc_1^2|\leq -4v+2=\frac{(8-12\mu)+(8\mu-8)\beta +2\beta^2}{(2-\beta)^2}.
\end{align*}
\noindent{\bf Case 2.} If $0 \leq v \leq 1$ (i.e., $0 \leq \mu \leq (2-\beta)^2/(3-2\beta)$), then by Lemma \ref{lem-3.1}, we see that
\begin{align*}
	|c_2-vc_1^2|\leq 2.
\end{align*}
\noindent{\bf Case 3.} If $v > 1$ (i.e., $\mu > (2-\beta)^2/(3-2\beta)$), then in view of Lemma \ref{lem-3.1}, we have
\begin{align*}
	|c_2-vc_1^2|\leq 4v-2=\frac{-(8-12\mu)-(8\mu-8)\beta -2\beta^2}{(2-\beta)^2}.
\end{align*}
Summarizing the above cases, we obtain the desired inequality. \vspace{2mm}
	
In the cases where $\mu < 0$ or $\mu >(2-\beta)^2/(3-2\beta)$, we consider the function $f_1 \in \mathcal{A}_{\beta}$ defined by
\begin{align*}
	f_1(z)=z\left(1+2\left(_2F_1\left[1,\frac{1}{1-\beta},\frac{2-\beta}{1-\beta},z\right]\right)\right)= z+ \sum_{n=2}^{\infty}\frac{2}{n- (n-1)\beta}z^n
\end{align*}
with $a_2={2}/{(2-\beta)}$, $a_3={2}/{(3-2\beta)}$. A straightforward calculation yields
\begin{align*}
	|a_3 -\mu a^2_2|=\;\vline\frac{(8-12\mu)+(8\mu-8)\beta +2\beta^2}{(2-\beta)^2(3-2\beta)}\vline=\pm\frac{(8-12\mu)+(8\mu-8)\beta +2\beta^2}{(2-\beta)^2(3-2\beta)},
\end{align*} 
the bound in the result is sharp. \vspace{1.2mm}
	
The inequality is sharp for $0 \leq \mu \leq (2-\beta)^2/(3-2\beta)$, with the extremal function $f_2 \in \mathcal{A}_{\beta}$ given by
\begin{align}\label{eq-3.4}
	\beta\frac{f_2(z)}{z}+(1-\beta)f^{\prime}_2(z)=\frac{1+z^2}{1-z^2}
\end{align}
with $a_2=0$, $a_3=2/(3-2\beta)$. The sharpness of the result follows by straightforward calculation. This concludes the proof.
\end{proof}
Note that for $\mu = 1$, the expression reduces to the classical Hankel determinant $H_{2,1}(f) = a_3 - a_2^2$; consequently, the following result is obtained as an immediate corollary of Theorem \ref{Th-3.1}.
\begin{cor}
If $f\in\mathcal{A}_{\beta}$ is of the form defined in \eqref{Eq-2.1}, then 
\begin{align*}
	|a_3 - a^2_{2}|\leq\frac{2}{(3-2\beta)}.
\end{align*}
The inequality is sharp for the function $f_2\in\mathcal{A}_{\beta}$ given by \eqref{eq-3.4}. 
\end{cor}

\section{\bf Moduli differences of logarithmic coefficients for functions and inverse functions in the class $\mathcal{A}_{\beta}$}
In 1985, de Branges \cite{Branges-AM-1985} solved the famous Bieberbach conjecture by proving that for any function $f\in\mathcal{S}$ of the form \eqref{Eq-2.1}, the inequality $ |a_n|\leq n $ holds for all $ n\geq 2 $, with equality attained by the Koebe function $ k(z):=z/(1-z)^2 $ or its rotations. This result naturally led to the question of whether the inequality $||a_{n+1}|-|a_n||\leq 1$ holds for $f\in\mathcal{S}$ when $ n\geq 2$. This problem was first studied by Goluzin \cite{Goluzin-1946} in an attempt to solve the Bieberbach conjecture. Later, in $1963$, Hayman \cite{Hayman-1963} showed that for all $ f\in\mathcal{S} $, there exists an absolute constant $A\geq 1$ such that $||a_{n+1}|-|a_n||\leq A$. The best known estimate to date is $ A=3.61 $, due to Grinspan \cite{Grinspan-1976}. On the other hand, for the class $ \mathcal{S} $, the sharp bound is known only for $n=2$ (see \cite[Theorem 3.11]{Duren-1983}), namely
\begin{align*}
	-1\leq |a_3|-|a_2|\leq 1.029...
\end{align*}
Similarly, for functions $ f\in\mathcal{S}^* $, Pommerenke \cite{Ch. Pommerenke-1971} conjectured that $ ||a_{n+1}|-|a_n||\leq 1 $, which was subsequently proved by Leung \cite{Leung-BLMS-1978} in $1978$. For convex functions, Li and Sugawa \cite{Li-Sugawa-CMFT2017} investigated the sharp bounds of $|a_{n+1}|-|a_n|$ for $n\geq 2$, and established the sharp results for $n=2, 3$. \vspace{2mm}

\noindent{\bf Logarithmic coefficient:} Let
\begin{align}\label{eq-4.1}
	F_{f}(z):=\log\dfrac{f(z)}{z}=2\sum_{n=1}^{\infty}\gamma_{n}(f)z^n, \;\; z\in\mathbb{D},\;\;\log 1:=0,
\end{align}
be a logarithmic function associated with $f\in\mathcal{S}$. The numbers $\gamma_{n}:=\gamma_{n}(f)$ are called the logarithmic coefficients of $f$. Differentiating \eqref{eq-4.1} and using \eqref{Eq-2.1}, a simple computation shows that 
\begin{align}\label{eq-4.2}
		\gamma_{1}=\dfrac{1}{2}a_{2}\; \mbox{and}\; \gamma_{2}=\dfrac{1}{2} \left(a_{3} -\dfrac{1}{2}a^2_{2}\right).
\end{align}
\noindent{\bf Logarithmic inverse coefficients:} Let $F$ be the inverse function of $f\in\mathcal{S}$ defined in a neighborhood of the origin with the Taylor series expansion
\begin{align}\label{eq-4.3}
	F(w):=f^{-1}(w)= w+\sum_{n=2}^{\infty} A_n w^n,
\end{align}
where we may choose $|w|<1/4$, as we know that the famous Koebe’s $1/4$-theorem ensures that, for each univalent function $f$ defined in $\mathbb{D}$, its inverse $f^{-1}$ exists at least on a disc of radius $1/4$. Let $f(z)=z+ \sum_{n=2}^{\infty}a_nz^n$ be a function in class $\mathcal{S}$. Since $f(f^{-1}(w))=w$ and using \eqref{eq-4.3} we obtain
\begin{align}\label{eq-4.4}
	A_2= -a_2\; \mbox{and}\; A_3=-a_3 +2a^2_{2}.
\end{align}
The notation of the logarithmic coefficient of inverse of $f$ has been studied by Ponnusamy \textit{et al.} \cite{Ponnusamy-Sharma-Wirths-RM-2018}. As with $f$, the logarithmic inverse coefficients $\Gamma_n:=\Gamma_n(F)$, $n\in\mathbb{N}$, of $F$ are defined by the equation
\begin{align}\label{eq-4.5}
	F_{f^{-1}}(w):=\log\left(\frac{f^{-1}(w)}{w}\right)=2\sum_{n=1}^{\infty} \Gamma_n(F) w^n \;\;\;\; \mbox{for} \;\;|w|<1/4.
\end{align}
By differentiating \eqref{eq-4.5} together with \eqref{eq-4.3}, using \eqref{eq-4.4} and then equating coefficients, we obtain
\begin{align}\label{eq-4.6}
	\Gamma_1=-\dfrac{1}{2}a_2\; \mbox{and}\; \Gamma_2=-\dfrac{1}{2}\left(a_3 -\dfrac{3}{2}a^2_{2}\right).
\end{align}
Inverse functions have been studied by several authors from different perspectives (see, for instance, \cite{Sim-Thomas-BAMS-2022, Sim-Thomas-S-2020}). Recently, Sim and Thomas \cite{Sim-Thomas-BAMS-2022} obtained sharp upper and lower bounds for the difference of the moduli of successive inverse coefficients for certain subclasses of univalent functions. Motivated by this prior research, including the recent works \cite{Allu-Shaji-JMAA-2025, Brown-IJMA-2010,Mandal-LMJ-2024}, the present paper focuses on determining sharp lower and upper bounds for $ |\gamma_2|-|\gamma_1|$ and $|\Gamma_2|-|\Gamma_1|$ for functions belonging to the class $\mathcal{A}_{\beta}$. \vspace{2mm}

To prove the following theorems, we rely on Lemma \ref{lem-4.1}, which plays a central role. For convenience, we restate it below.
\begin{lem}\cite{Sim-Thomas-S-2020}\label{lem-4.1}
Let $B_1$, $B_2$ and $B_3$ be numbers such that $B_1>0$, $B_2\in\mathbb{C}$ and $B_3\in\mathbb{R}$. Let $p\in\mathcal{P}$ of the form \eqref{Eq-3.1}. Define $\Psi_{+}(c_1,c_2)$ and $\Psi_{-}(c_1,c_2)$ by
\begin{align*}
	\Psi_{+}(c_1,c_2)=|B_2 c^2_1 +B_3 c_2| -|B_1 c_1|,
\end{align*}
and
\begin{align*}
	\Psi_{-}(c_1,c_2)=-\Psi_{+}(c_1,c_2).
\end{align*}
Then
\begin{align}\label{eq-4.7}
	\Psi_{+}(c_1,c_2)\leq
	\begin{cases}
		|4B_2 +2B_3|-2B_1, \;\;\;\;\mbox{if}\;\;|2B_2 +B_3|\geq |B_3|+ B_1,\vspace{2mm} \\ 2|B_3|,\hspace{2.8cm}\;\mbox{otherwise},
	\end{cases}
\end{align}
and
\begin{align}\label{eq-4.8}
	\Psi_{-}(c_1,c_2)\leq
	\begin{cases}
		2B_1 -B_4, \hspace{2.8cm}\mbox{if}\;\; B_1\geq B_4 +2|B_3|, \vspace{2mm} \\ 2B_1 \sqrt{\dfrac{2|B_3|}{B_4+2|B_3|}}, \hspace{1.3cm}\;\mbox{if}\;\;B^2_1\leq 2|B_3|(B_4 +2|B_3|), \vspace{2mm} \\ 2|B_3| +\dfrac{B^2_1}{B_4+2|B_3|}, \hspace{1cm}\;\mbox{otherwise},
	\end{cases}
\end{align}
where $B_4=|4B_2 +2B_3|$. All inequalities in \eqref{eq-4.7} and \eqref{eq-4.8} are sharp.
\end{lem}
By Lemma \ref{lem-4.1}, we derive the following sharp estimates for the difference $|\gamma_2| - |\gamma_1|$, where $\gamma_1$ and $\gamma_{2}$ are the logarithmic coefficients of functions in the class $\mathcal{A}_{\beta}$.
\begin{thm}\label{th-4.1}
Let $0\leq\beta\leq 1$. If $f\in\mathcal{A}_{\beta}$ is given by \eqref{Eq-2.1} and $\gamma_{1}, \gamma_{2} $ are given by \eqref{eq-4.2}, then we have 
\begin{align*}
	-\frac{1}{\sqrt{5-6\beta+2\beta^2}}\leq|\gamma_2|-|\gamma_1|\leq \frac{1}{(3-2\beta)}.
\end{align*}
Both inequalities are sharp.
\end{thm}
\begin{proof}
In view of \eqref{eq-3.3} and \eqref{eq-4.2}, an elementary calculation shows that
\begin{align}\label{eq-4.9}
	\nonumber|\gamma_2|-|\gamma_1|&=\;\vline\; \frac{1}{2} a_{3} -\frac{1}{4}a^2_{2}\;\vline - \;\vline\; \frac{1}{2} a_{2}\;\vline \\&\nonumber=\;\vline\; \frac{1}{2}\cdot\frac{c_2}{(3-2\beta)} -\frac{1}{4}\cdot\frac{c^2_1}{(2-\beta)^2}\;\vline - \;\vline\; \frac{1}{2}\cdot\frac{c_1}{(2-\beta)}\;\vline\\&\nonumber =\;\vline\; -\frac{1}{4(2-\beta)^2}c^2_1 +\frac{1}{2(3-2\beta)}c_2\;\vline - \;\vline\; \frac{1}{2(2-\beta)} c_1\;\vline\\&\nonumber = \frac{1}{2(2-\beta)} \left(|B_2 c^2_1 +B_3c_2| - |B_1c_1| \right)\\& =\frac{1}{2(2-\beta)}\Psi_{+}(c_1,c_2),
\end{align}
where
\begin{align*}
	B_1:=1,\;\; B_2:=-\frac{1}{2(2-\beta)} \;\;\mbox{and}\;\; B_3:=\frac{(2-\beta)}{(3-2\beta)}.
\end{align*}
\noindent{\bf Estimate of the upper bound:} For the upper bound, a simple calculation shows that
\begin{align*}
	|2B_2 + B_3| = \frac{(1-\beta)^2}{(2-\beta)(3-2\beta)} \quad \text{and} \quad |B_3| + B_1 = \frac{5-3\beta}{3-2\beta}.
\end{align*}
Since the inequality $|2B_2 + B_3| \geq |B_3| + B_1$ fails for all $\beta \in [0, 1]$ (as it would require $2\beta^2 - 9\beta + 9 \leq 0$), we apply Lemma \ref{lem-4.1} to obtain
\begin{align*}
	\Psi_{+}(c_1,c_2)\leq 2|B_3|=\frac{2(2-\beta)}{(3-2\beta)}.
\end{align*}
Consequently, applying \eqref{eq-4.9}, we obtain
\begin{align}\label{eq-4.10}
	|\gamma_2|-|\gamma_1|\leq \frac{1}{(3-2\beta)}.
\end{align}
Equality in \eqref{eq-4.10} is attained for the function $f \in \mathcal{A}$ defined in \eqref{eq-3.2}, where
\begin{align*}
	p(z)=\frac{1+z^2}{1 -z^2},\;\;\;\; z\in\mathbb{D}.
\end{align*}
For $c_1=0$ and $c_2=2$, equality is attained in \eqref{eq-4.10}, confirming the sharpness of the result.\vspace{2mm}
	
\noindent{\bf Estimate of the lower bound:} For the lower bound, we have
\begin{align*}
	B_4=|4B_2 +2B_3|=\frac{2(1-\beta)^2}{(2-\beta)(3-2\beta)} \hspace{0.5cm} \mbox{and}\hspace{0.5cm} B_4 +2|B_3|=\frac{10-12\beta+4\beta^2}{(2-\beta)(3-2\beta)}.
\end{align*}
It is easy to see that, the inequality $B_1\geq B_4 +2|B_3|$ implies that $4+19\beta+2\beta^2\leq0$ for all $\beta\in[0,1]$, which is not true. Again, we have
\begin{align*}
	2|B_3|(B_4 +2|B_3|) =\frac{4(5-6\beta+2\beta^2)}{(3-2\beta)^2}.
\end{align*}
Thus, the condition $B^2_1\leq 2|B_3|(B_4 +2|B_3|)$ follows that $11-12\beta+4\beta^2\geq 0$ for all $\beta\in[0,1]$, which is true. Thus, by Lemma \ref{lem-4.1}, we have
\begin{align*}
	\Psi_{-}(c_1,c_2)\leq 2B_1 \sqrt{\dfrac{2|B_3|}{B_4+2|B_3|}} =\frac{2(2-\beta)}{\sqrt{5-6\beta+2\beta^2}}.
\end{align*}
Clearly, we observe that
\begin{align*}
	\Psi_{+}(c_1,c_2)=-\Psi_{-}(c_1,c_2)\geq -\frac{2(2-\beta)}{\sqrt{5-6\beta+2\beta^2}}.
\end{align*}
Hence, from \eqref{eq-4.9}, we conclude that
\begin{align}\label{eq-4.11}
	|\gamma_2|-|\gamma_1|\geq -\frac{1}{\sqrt{5-6\beta+2\beta^2}}.
\end{align}
The inequality \eqref{eq-4.11} is sharp for the function $f\in\mathcal{A}$ given by \eqref{eq-3.2} with
\begin{align*}
	p(z)=\frac{1-z^2}{1-\frac{2(2-\beta)}{\sqrt{5-6\beta+2\beta^2}}z +z^2},\;\;\;\; z\in\mathbb{D}.
\end{align*}
This completes the proof.
\end{proof}
For $\beta = 0$, Theorem \ref{th-4.1} reduces to the following result for the logarithmic coefficients of the class of functions with bounded turning.
\begin{cor}
If $f\in\mathcal{R}$ is given by \eqref{Eq-2.1} and $\gamma_{1}, \gamma_{2} $ are given by \eqref{eq-4.2}, then we have 
\begin{align*}
	-\frac{1}{\sqrt{5}}\leq|\gamma_2|-|\gamma_1|\leq \frac{1}{3}.
\end{align*}
Both inequalities are sharp.
\end{cor}
By Lemma \ref{lem-4.1}, we derive the following sharp estimates for the difference $|\Gamma_2| - |\Gamma_1|$, where $\Gamma_1$ and $\Gamma_2$ are the logarithmic coefficients of the inverse function for functions in the class $\mathcal{A}_{\beta}$.
\begin{thm}\label{th-4.2}
Let $0\leq\beta\leq 1$. If $f\in\mathcal{A}_{\beta}$ is given by \eqref{Eq-2.1} and $ \Gamma_{1}, \Gamma_{2} $ are given by \eqref{eq-4.6}, then we have 
\begin{align*}
	|\Gamma_2|-|\Gamma_1|\leq 
	\begin{cases}
		\dfrac{1}{(3-2\beta)} \hspace{2.3cm}\mbox{for}\;\; 0\leq\beta<1, \vspace{2mm}\\ \dfrac{(-1+5\beta -3\beta^2)}{(2-\beta)^2 (3-2\beta)} \hspace{0.8cm} \mbox{for}\;\;\beta=1,
	\end{cases}
\end{align*}
and
\begin{align*}
	|\Gamma_2|-|\Gamma_1|\geq -\frac{1}{\sqrt{3(3-2\beta)}}.
\end{align*}
All the inequalities are sharp.
\end{thm}
\begin{proof}
In view of \eqref{eq-3.3} and \eqref{eq-4.6}, an elementary calculation shows that
\begin{align}\label{eq-4.12}
	\nonumber|\Gamma_2|-|\Gamma_1|&=\;\vline\; -\frac{1}{2}a_{3} +\frac{3}{4}a^2_{2}\;\vline - \;\vline\; -\frac{1}{2} a_{2}\;\vline\\&\nonumber=\;\vline\; -\frac{1}{2}\cdot\frac{c_2}{(3-2\beta)} +\frac{3}{4}\cdot \frac{c^2_1}{(2-\beta)^2}\;\vline - \;\vline\; -\frac{1}{2}\cdot\frac{c_1}{(2-\beta)}\;\vline\\&\nonumber =\;\vline\; \frac{3}{4(2-\beta)^2}c^2_1 -\frac{1}{2(3-2\beta)}c_2\;\vline - \;\vline\; \frac{1}{2(2-\beta)} c_1\;\vline\\&\nonumber = \frac{1}{2(2-\beta)} \left(|B_2 c^2_1 +B_3c_2| - |B_1c_1| \right)\\& =\frac{1}{2(2-\beta)}\Psi_{+}(c_1,c_2),
\end{align}
where
\begin{align*}
	B_1=1,\;\; B_2=\frac{3}{2(2-\beta)} \;\;\mbox{and}\;\; B_3=- \frac{(2-\beta)}{(3-2\beta)}.
\end{align*}
\noindent{\bf Estimate of the upper bound:} For the upper bound, we see that
\begin{align*}
	|2B_2 +B_3|=\frac{(5-2\beta-\beta^2)}{(2-\beta)(3-2\beta)} \hspace{0.5cm} \mbox{and}\hspace{0.5cm} |B_3|+ B_1=\frac{(5-3\beta)}{(3-2\beta)}.
\end{align*}
From the inequality $|2B_2 +B_3|\geq |B_3|+ B_1$, it follows that $5-9\beta+4\beta^2\leq 0$, which is true if and only if $\beta=1$, by Lemma \ref{lem-4.1}, we obtain
\begin{align*}
	\Psi_{+}(c_1,c_2)\leq |4B_2 +2B_3|-2B_1=-\frac{2(1-5\beta +3\beta^2)}{(2-\beta)(3-2\beta)}.
\end{align*}
For $0\leq\beta<1$, applying Lemma \ref{lem-4.1}, we obtain
\begin{align*}
	\Psi_{+}(c_1,c_2)\leq 2|B_3|=\frac{2(2-\beta)}{(3-2\beta)}.
\end{align*}
As a result, from \eqref{eq-4.12}, we obtain
\begin{align}\label{eq-4.13}
	|\Gamma_2|-|\Gamma_1|\leq 
	\begin{cases}
		\dfrac{1}{(3-2\beta)} \hspace{2.3cm}\mbox{for}\;\; 0\leq\beta<1, \vspace{2mm}\\ \dfrac{(-1+5\beta -3\beta^2)}{(2-\beta)^2 (3-2\beta)} \hspace{0.8cm} \mbox{for}\;\;\beta=1.
	\end{cases}
\end{align}
The first inequality in \eqref{eq-4.13} is sharp for the function $f\in\mathcal{A}$ given by \eqref{eq-3.2} with
\begin{align*}
	p(z)=\frac{1+z^2}{1 -z^2},\;\;\;\; z\in\mathbb{D}.
\end{align*}
In this case, we have $c_1=0$ and $c_2=2$, which establishes the sharpness of the first inequality in \eqref{eq-4.13}. \vspace{2mm}

The second inequality in \eqref{eq-4.13} is sharp for the function $f\in\mathcal{A}$ given by \eqref{eq-3.2} with
\begin{align*}
	p(z)=\frac{1+z}{1 -z},\;\;\;\; z\in\mathbb{D}.
\end{align*}
In this case, $c_1=2$ and $c_2=2$, which demonstrates the sharpness of the second inequality in \eqref{eq-4.13}. \vspace{2mm}
	
\noindent{\bf Estimate of the lower bound:} For the lower bound, we see that
\begin{align*}
	B_4=|4B_2 +2B_3|=\frac{2(5-2\beta-\beta^2)}{(2-\beta)(3-2\beta)} \hspace{0.5cm} \mbox{and}\hspace{0.5cm} B_4 +2|B_3|=\frac{6}{(2-\beta)}.
\end{align*}
It follows that $B_1\not\geq B_4 +2|B_3|$ for all $\beta\in[0,1]$. Again, we have
\begin{align*}
	2|B_3|(B_4 +2|B_3|)=\frac{12}{(3-2\beta)},
\end{align*} 
the condition $B^2_1\leq 2|B_3|(B_4 +2|B_3|)$ implies that $9+2\beta\geq 0$ for all $\beta\in[0,1]$, which is true. Thus, by Lemma \ref{lem-4.1}, we have
\begin{align*}
	\Psi_{-}(c_1,c_2)\leq 2B_1 \sqrt{\dfrac{2|B_3|}{B_4+2|B_3|}} =\frac{2(2-\beta)}{\sqrt{3(3-2\beta)}}.
\end{align*}
One can easily observe that
\begin{align*}
	\Psi_{+}(c_1,c_2)=-\Psi_{-}(c_1,c_2)\geq -\frac{2(2-\beta)}{\sqrt{3(3-2\beta)}}.
\end{align*}
Hence, from \eqref{eq-4.12}, we derive that
\begin{align}\label{eq-4.14}
	|\Gamma_2|-|\Gamma_1|\geq -\frac{1}{\sqrt{3(3-2\beta)}}.
\end{align}
The inequality \eqref{eq-4.14} is sharp for the function $f\in\mathcal{A}$ given by \eqref{eq-3.2} with
\begin{align*}
	p(z)=\frac{1+\frac{2(2-\beta)}{\sqrt{3(3-2\beta)}}z +z^2}{1 -z^2},
\end{align*}
which completes the proof.
\end{proof}
Setting $\beta = 0$ in Theorem \ref{th-4.2}, we obtain the following result for the logarithmic coefficients of an inverse function in the class of bounded turning functions.
\begin{cor}
If $f\in\mathcal{R}$ is given by \eqref{Eq-2.1} and $ \Gamma_{1}, \Gamma_{2} $ are given by \eqref{eq-4.6}, then we have 
\begin{align*}
	-\frac{1}{3}\leq|\Gamma_2|-|\Gamma_1|\leq \frac{1}{3}.
\end{align*}
Both inequalities are sharp.
\end{cor}\vspace{5mm}

\noindent{\bf Acknowledgment:} The authors express their sincere gratitude to Prof. Swadesh Kumar Sahoo for several helpful discussions and suggestions. Additionally, we would like to thank the reviewers for their thoughtful feedback, which significantly enhanced the clarity and presentation of the manuscript.\\

\noindent{\bf Author Contributions:} Both authors actively contributed to the research presented in this paper and reviewed the manuscript. \\

\noindent\textbf{Compliance of Ethical Standards:}\\

\noindent\textbf{Conflict of interest} The authors declare that there is no conflict  of interest regarding the publication of this paper.\vspace{1.5mm}

\noindent\textbf{Data availability statement}  Data sharing not applicable to this article as no datasets were generated or analyzed during the current study.

\end{document}